\documentclass[11pt]{amsart}
\usepackage{hyperref}
\usepackage[capitalize, nameinlink]{cleveref}
\crefname{theorem}{Theorem}{Theorems}
\crefname{proposition}{Proposition}{Propositions}
\crefname{observation}{Observation}{Observations}
\crefname{lemma}{Lemma}{Lemmas}
\crefname{claim}{Claim}{Claims}
\crefname{problem}{Problem}{Problems}
\crefname{conjecture}{Conjecture}{Conjectures}
\crefname{question}{Question}{Questions}
\crefname{example}{Example}{Examples}
\crefname{fact}{Fact}{Facts}

\usepackage{amssymb}
\usepackage[bbgreekl]{mathbbol}
\usepackage{graphicx}
\usepackage{ifthen}
\usepackage{pict2e}
\usepackage{xargs}
\usepackage{xspace}
\usepackage{xcolor}
\usepackage{pgf,tikz}
\usepackage{pgfplots}
\usepackage{environ}
\usepackage{caption}
\usetikzlibrary{arrows}
\usepackage{enumitem}
\usepackage{xifthen}
\usepackage{comment}
\usepackage{soul}

\DeclareFontFamily{U}{mathx}{\hyphenchar\font45}
\DeclareFontShape{U}{mathx}{m}{n}{
	<5> <6> <7> <8> <9> <10>
	<10.95> <12> <14.4> <17.28> <20.74> <24.88>
	mathx10
}{}
\DeclareSymbolFont{mathx}{U}{mathx}{m}{n}
\DeclareMathSymbol{\bigtimes}{1}{mathx}{"91}

\newcounter{dummy}
\makeatletter
\newcommand\myitem[1][]{\item[#1]\refstepcounter{dummy}\def\@currentlabel{#1}}
\makeatother

\makeatletter
\newsavebox{\measure@tikzpicture}
\NewEnviron{scaletikzpicturetowidth}[1]{%
	\def\tikz@width{#1}%
	\begin{lrbox}{\measure@tikzpicture}%
		\BODY
	\end{lrbox}%
	\pgfmathparse{#1/\wd\measure@tikzpicture}%
	\BODY
}
\makeatother

\DeclareSymbolFontAlphabet{\mathbb}{AMSb}

\newcommand{\thistheoremname}{}
\newtheorem*{genericthm*}{\thistheoremname}
\newenvironment{namedthm*}[1]
{\renewcommand{\thistheoremname}{#1}%
	\begin{genericthm*}}
	{\end{genericthm*}}

\newcommand{\Bairespace}[1][]{
	\ifthenelse{\equal{#1}{}}{\functions{\N}{\N}}{\functions{#1}{\N}}
}
\newcommand{\bbL}{\mathbb{L}}
\newcommand{\bbX}{\mathbb{X}}

\newcommand{\Cantorspace}[1][]{
	\ifthenelse{\equal{#1}{}}{\functions{\N}{2}}{\functions{#1}{2}}
}

\newcommandx{\concatenation}[2][1 = undefined, 2 = undefined]{
	\ifthenelse{\equal{#1}{undefined}}{{}\smallfrown}{
		\ifthenelse{\equal{#2}{undefined}}{\bigoplus #1}{\bigoplus_{#1} #2}
	}
}

\newcommandx{\functions}[3][3 =]{
	\ifthenelse{\equal{#3}{}}{#2^{#1}}{#2_{#3}^{#1}}
}

\newcommand{\Gzero}[1][]{
	\ifthenelse{\equal{#1}{}}
	{\mathbb{G}_0}
	{\mathbb{G}_{0,n}}
}
\newcommandx{\Hzero}[2][2 = undefined]{
	\ifthenelse{\equal{#2}{undefined}}
	{\mathbb{H}_{#1}}
	{\mathbb{H}_{#1, #2}}
}
\newcommandx{\intersection}[2][1 =, 2 =]{
	\ifthenelse{\equal{#1}{}}{\cap}{
		\ifthenelse{\equal{#2}{}}{\bigcap #1}{{\bigcap_{#1} #2}}
	}
}
\newcommand{\Lzero}[1][]{\ifthenelse{\equal{#1}{}}{\bbL_0}{L_{0, #1}}}
\newcommand{\Lzerospace}[1][]{\ifthenelse{\equal{#1}{}}{\bbX_0}{X_{0, #1}}}

\newcommand{\modulo}[1]{\ (\text{mod } 2)}
\newcommand{\N}{\mathbb{N}}

\newcommandx{\product}[2][1 =, 2 =]{
	\ifthenelse{\equal{#1}{}}{\times}{
		\ifthenelse{\equal{#2}{}}{\prod #1}{{\prod_{#1} #2}}
	}
}

\newcommandx{\sequence}[2][2 = undefined]{
	\ifthenelse{\equal{#2}{undefined}}{(#1)}{
		(#1)_{#2}
	}
}

\newcommandx{\set}[2][2 = undefined]{
	\ifthenelse{\equal{#2}{undefined}}{\{ #1 \}}{
		\{ #1 \suchthat #2 \}
	}
}
\newcommandx{\sets}[3][3 =]{
	\ifthenelse{\equal{#3}{}}{[#2]^{#1}}{[#2]^{#1}_{#3}}
}

\newcommand{\suchthat}{\mid}

\newcommandx{\union}[2][1 =, 2 =]{
	\ifthenelse{\equal{#1}{}}{\cup}{
		\ifthenelse{\equal{#2}{}}{\bigcup #1}{{\bigcup_{#1} #2}}
	}
}

\newtheorem{theorem}{Theorem}[section]
\newtheorem{lemma}[theorem]{Lemma}

\newtheorem{corollary}[theorem]{Corollary}

\newtheorem{problem}[theorem]{Problem}

\theoremstyle{definition}

\newtheorem{definition}[theorem]{Definition}

\numberwithin{equation}{section}

\newcommand{\bd}{\begin{definition}}
	\newcommand{\ed}{\end{definition}}

\DeclareMathOperator{\dist}{dist}
\DeclareMathOperator{\didistance}{didist}
\newcommand{\mc}{\mathcal}

\newcommand{\distance}[3]{\ifthenelse{\isempty{#3}}{\dist(#1,#2)}{\dist^{#3}(#1,#2)}}
\newcommand{\didist}[3]{\ifthenelse{\isempty{#3}}{\didistance(#1,#2)}{\didistance^{#3}(#1,#2)}}
\newcommand{\digraph}[3]{\ifthenelse{\equal{#1}{b}}{\mathbb{#2}_{#3}}
	{{#2}_{#3}}}
\newcommand{\linegraph}[3]{\ifthenelse{\equal{#1}{b}}{\mathbb{#2}_{#3}}
	{#2_{#3}}}

\newcommand{\underlyingspace}[3]{\ifthenelse{\equal{#1}{b}}{\mathbb{#2}_{#3}}
	{#2_{#3}}}
\newcommand{\distanceset}[2]{\ifthenelse{\isempty{#2}}{D(#1)}{D^{#2}(#1)}}

\def\vertex{
	\begin{tikzpicture}[baseline=-0.6ex]
	\filldraw 
	(0,0) circle (1.1pt);
	\end{tikzpicture}
}

\def\edge{
	\begin{tikzpicture}[baseline=-0.6ex]
	\filldraw 
	(0,-0.1) circle (1.1pt) -- (0,0.1) circle (1.1pt);
	\end{tikzpicture}
}

\def\nbr{
	\begin{tikzpicture}[baseline=-0.6ex]
	\filldraw 
	(0,0.1) circle (1.1pt);
	\draw (0,-0.1) -- (0,0.1);
	\end{tikzpicture}
}
\author{Endre Cs\'oka}

\address{Alfr\'ed R\'enyi Institute of Mathematics, Reáltanoda u. 13-15., 1053 Budapest, Hungary}
\email{csokaendre@gmail.com}

\author{Zolt\'an Vidny\'anszky}
\address{E\"otv\"os Lor\'and University, Institute of Mathematics, P\'azm\'any P\'eter stny. 1/C, 1117 Budapest, Hungary}
\email{zoltan.vidnyanszky@ttk.elte.hu}
\thanks{The first author was supported by National Research, Development and Innovation Office (NKFIH) grant KKP no. 138270. The second author was supported by Hungarian Academy of Sciences Momentum Grant no. 2022-58 and National Research, Development and Innovation Office (NKFIH) grants no.~113047, ~129211. }
\begin{document}

	\thanks{}
	
	
	
	\title[]{FIID homomorphisms and entropy inequalities}

	
	
	\maketitle
	
	\begin{abstract}
		
	We investigate the existence of FIID homomorphisms from regular trees to finite graphs. Using entropy inequalities we show that there are graphs with arbitrarily large chromatic number to which there is no FIID homomorphism from a $3$-regular tree.

	\end{abstract}
	\section{Introduction}
	In the past decade, a large amount of work has been aimed at the understanding of \emph{factor of i.i.d.\ (FIID)} processes in general and FIID tree colorings in particular (see, e.g., \cite{ArankaFeriLaci,backhausz2018large,BackVir,BGHV,csoka2015invariant,GamarnikSudan,HarangiVirag,HoppenWormald,RahVir,LyonsTrees,Timar1}). Roughly speaking (for the formal definition, see Section \ref{s:prel}), we consider a $d$-regular tree on which we want to solve a graph theoretic problem such as vertex coloring, perfect matching, or finding a maximal independent set, etc. The solution has to respect the automorphisms of the tree but it is allowed to use uniformly distributed independent random reals at every vertex. 
    One of the most important feautre of FIID processes that they can be applied with an arbitrarily small error to $d$-regular graphs with large essential girth (i.e., to $d$-regular graphs which contain a small number of cycles). Random $d$-regular graphs have large essenial girth, and the optimum of many optimization problems on these random graphs are often attained by FIID processes.
 
	Motivated by the question of proper colorings, one can formulate the next general problem (here, $T_d$ stands for the $d$-regular tree). 	
		
	\begin{problem}
		\label{p:main}
		Is it possible to characterize the finite graphs $H$ to which $T_d$ admits an FIID homomorphism?
	\end{problem}
	
	Obtaining an affirmative answer, i.e., a complete characterization, seems to be an overly optimistic goal: despite a significant amount of effort, already the existence of homomorphisms of $T_3$ to concrete, small 
    graphs is still open. For example, a negative answer in the case of $H=C_5$ would suggest that large (finite) $3$-regular random graphs do not admit a homomorphism to $C_5$, answering the notoriously hard pentagon problem of Ne\v{s}et\v{r}il \cite{nesetril1999aspects}. 
	
	In this paper, we provide a simple argument based on entropy inequalities to exclude homomorphisms to certain graphs, showing the following. 
	
	\begin{theorem}
		\label{t:main} 
		Let $r \in \N$ be arbitrary. There exists a constant $C_r \in \N$ such that there is no FIID homomorphism from $T_3$ to any $r$-regular graph $H$ with girth at least $C_r$. 
	\end{theorem}

	\subsection{Discussion and some corollaries}
	
	 \cite[Problem 8.12]{kechris2015descriptive} asks for the characterization of finite graphs $H$ to which each $d$-regular acyclic Borel graph admits a Borel homomorphism. 
	As FIID processes on $T_3$ can be considered as measurable maps from a fixed $3$-regular Borel graph, our theorem gives some information about this problem as well. 
	\begin{corollary}
		There exists a $3$-regular acyclic Borel graph, which does not admit a Borel homomorphism to any finite $r$-regular graph $H$ with girth at least $C_r$.   
	\end{corollary}

   Connected to the problem of Kechris and Marks and a generalization of Marks' determinacy method \cite{marksdet}, in \cite{brandt2021local} (see also \cite{brandt2021homomorphism}) it was shown that there are graphs $H$ with chromatic number $2d-2$ such that there is no Borel homomorphism from a $d$-regular acyclic Borel graph to $H$. By classical results of Bollob\'as \cite{bollobas1981independence}, for any $g$ there is an $r$-regular graph $H$ with girth at least $g$ and $\chi(H) \geq \frac{r}{2\log r}$.  Thus, using our main result, we can eliminate the $2d-2$ upper bound discussed above:
	\begin{corollary}
		There are graphs of arbitrarily large chromatic number to which there is no FIID homomorphism from $T_3$. 
	\end{corollary}

	By the correspondence developed in \cite{brandt2021local} and \cite{bernshteyn2023distributed} (see also \cite{RahVir}) between descriptive combinatorics and local algorithms, we get the following.  
	
	\begin{corollary}
		There is no $o(\log n)$-round randomized local algorithm which outputs a homomorphism from an acyclic graph with degrees $\leq 3$ of size $n$ to an $r$-regular graph of girth at least $C_r$.  
	\end{corollary}
	
	Thus, entropy inequalities provide a novel way of proving local model lower bounds for homomorphism problems on trees, and these bounds, to our knowledge, could not be attained by Brandt's round elimination method \cite{brandt19automatic}. 
	
	\textbf{Acknowledgments.} We are very grateful to \'Agnes Backhausz, Jan Greb\'ik and Viktor Harangi for their insightful suggestions. 
	
	\section{Preliminaries}
	\label{s:prel}

	We will describe two equivalent ways of talking about FIID processes. The classical definition is as follows. Let $\Gamma$ be a group acting on measurable spaces $(X,\mc{B})$ and $(Y,\mc{C})$. A map $\phi:X \to Y$ is \emph{$\Gamma$-equivariant} if for all $x \in X$ and $\gamma \in \Gamma$ we have 
	\[\phi(\gamma \cdot x)=\gamma \cdot \phi(x).\] 
	
	A measurable equivariant map is called a \emph{$\Gamma$-factor}. If $\mu$ is a measure on $(X,\mathcal{B})$ and $\phi:X \to Y$ is a $\Gamma$-factor, the measure $\phi_*\mu$ is called a \emph{$\Gamma$-factor of $\mu$}.
	
	Let $T_d$ stand for the $d$-regular tree and $Aut(T_d)$ be its automorphism group. If $S$ is a set, $Aut(T_d)$ acts on $S^{V(T_d)}$ by the \emph{(left-)shift action}, that is,
	\[(\gamma \cdot x) (v):=x(\gamma^{-1} \cdot v).\]
	
	Let $\mu$ be the uniform measure on the space $[0,1]^{V(T_d)}$, that is, the product of the Lebesgue measures on each copy of $[0,1]$. An \emph{FIID process} is an $Aut(T_d)$-factor of $\mu$ on $S^{V(T_d)}$, where $S$ is finite and $[0,1]^{V(T_d)}$ and $S^{V(T_d)}$ are considered with the algebra of Lebesgue-measurable sets. An \emph{FIID homomorphism to a graph $H$} is an FIID process on $V(H)^{V(T_d)}$ such that for $\mu$ almost every $x \in [0,1]^{V(T_d)}$ the image $\phi(x)$ is a homomorphism from $T_d$ to $H$.  
	
	A slightly different view is the following. One can define a graph on the space $[0,1]^{V(T_d)}$ as follows: fix a ``root", i.e., a distinguished element $v_0 \in V(T_d)$. Let $\{x,y\}$ form an edge in $\mc{G}$ if there is an automorphism $\psi \in Aut(T_d)$ such that $x \circ \psi=y$ and $(\psi(v_0),v_0) \in E(T_d)$. It is easy to see that there is a $\mu$ co-null Borel set on which $\mc{G}$ is $d$-regular, acyclic and Borel. Moreover, there is a correspondence between FIID homomorphisms to some $H$ in the above sense and Borel homomorphisms of $\mc{G}$ to $H$ defined on a co-null set: indeed, evaluating the factor map at the root yields the latter kind of homomorphism, and conversely, given an a.e. defined Borel homomorphism $h$ from $\mc{G}$ to $H$ we get a factor map $\phi$ by letting $\phi(x)(v)=h(\gamma^{-1} \cdot x)$, where $\gamma$ is arbitrary in $Aut(T_d)$ with $\gamma(v_0)=v$.

	\subsection{Entropy inequalities.} Recall that if $X$ is a discrete random variable, the \emph{entropy of $X$} is defined by 
	\[h(X)=-\sum_x \mathbb{P}(X=x)\ln  \mathbb{P}(X=x),\] 
    (if $\mathbb{P}(X=x)=0$ then the corresponding product is considered to be $0$), 
	while for two such variables $X,Y$ on the same space we define the \emph{conditional entropy} by
	\[h(X|Y)=-\sum_{x,y}\mathbb{P}(X=x,Y=y)\ln \frac{\mathbb{P}(X=x,Y=y)}{\mathbb{P}(Y=y)} .\]
	It is clear that 
	\[h(X,Y)=h(X)+h(X|Y),\]
	where $h(X,Y)$ stands for the \emph{joint entropy} of $X$ and $Y$, that is, the entropy of the joint distribution of $X$ and $Y$. 
	
	Now observe that for any $F \subset V(T_d)$ finite an FIID homomorphism to $H$ gives rise to an $V(H)^{F}$ valued discrete random variable. It follows from the $Aut(T_d)$ equivariance that this variable only depends on the $Aut(T_d)$-orbit of $F$. Thus, for an FIID homomorphism we can consider the entropy of the image in $H$ of a random vertex, random edge, random neighbor, denote these by $h_f(\vertex),h_f(\edge),h_f(\nbr),$ respectively.

	The following fundamental result is shown in \cite{backhausz2018large,bowen2010measure,RahVir} in various generalities (see also \cite{backhausz2018entropy}). 
	\begin{theorem}
		\label{t:edgevertex}
		For any FIID homomorphism $f$ (or more generally, FIID process) on $T_3$ we have
		\[\frac{4}{3}h_f(\vertex) \leq h_f(\edge).\]
		
	\end{theorem}

	The next theorem follows from the fact that large random $3$-regular graphs have independence ratio $<\frac{1}{2}$ (\cite{bollobas1981independence}) together with the observation that FIID processes can be ``emulated" on random regular graphs (see \cite[Section 4]{LyonsTrees}).
	
	\begin{theorem}
		\label{t:2col}
		The size of the largest FIID independent set in $T_3$ is bounded away from $\frac{1}{2}$. In particular, there is some $c_0>0$ such that if $f$ is a FIID partial $2$-coloring of $T_3$ then $\mu(dom(f))<1-c_0$. 
	\end{theorem}

	\section{Proof of the theorem}
		
		Now we state and prove a more precise version of our result. Let $C_r>r^{\frac{3}{c_0}}$, where $c_0$ is the constant from Theorem \ref{t:2col}. 
	\begin{theorem}
		Assume that $H$ is an $r$-regular graph with girth $ \geq C_r$. Then there is no FIID homomorphism from $T_3$ to $H$.  
	\end{theorem}

	\begin{proof}
			Assume that $f$ is a FIID homomorphism to the $r$-regular graph $H$. We start with estimating the vertex entropy. 
	\begin{lemma}
		\label{l:entropy}
	$h_f(\vertex) \leq 3\ln r$.  
	\end{lemma}

	\begin{proof}
		 Observe that
		 \[h_f(\edge)= h_f(\vertex)+h_f(\nbr|\vertex),\] by the relationship between conditional entropy and joint entropy discussed above.
		 Since $H$ is $r$-regular, given the image of a vertex, there are only $r$-many possibilities for choosing its neighbor. In particular, as the entropy is maximized when the probabilities of the outcomes are equal, the conditional entropy
		 $h_f(\nbr|\vertex)$ is bounded by $\ln r$. Thus,
		 \[h_f(\edge) \leq h_f(\vertex)+\ln r.\] Combining this with Theorem \ref{t:edgevertex}, i.e., with 
		  \[\frac{4}{3}h_f(\vertex) \leq h_f(\edge),\]
		 and 
		 rearranging, we get the desired inequality.
	\end{proof}

		Now we show that 
	\begin{lemma}
		\label{l:largepreimage}
		There exists a set $S \subset V(H)$ such that $|S|<C_r$ and $\mu(f^{-1}(S)) \geq 1-c_0$. 
		
	\end{lemma}

	\begin{proof}
		For $v \in V(H)$ let $p_v=\mu(f^{-1}(v))$. Take the set $S$ of vertices of $H$ having the first $C_r-1$ largest $p_v$ values. 
		Now, assume $\mu(f^{-1}(S))<1-c_0$. Then
		\[h_f(\vertex) \geq \sum_{v \not \in S}-p_v\ln p_v \geq -c_0 \ln \frac{1}{C_r},\]
		and by Lemma \ref{l:entropy} we have
		\[3\ln r \geq c_0 \ln C_r,\]
		so \[C_r \leq r^{\frac{3}{c_0}},\] contradicting the choice of $C_r$.  
	\end{proof}
	
	 	In order to finish the proof, observe that as the girth of $H$ is $\geq C_r$, the set $S$ induces an acyclic graph in $H$. Let $c$ be a $2$-coloring of this graph. Then $c \circ f|_{f^{-1}(S)}$ is a partial FIID $2$-coloring of $T_3$ with domain of size $\geq 1-c_0$ contradicting Theorem \ref{t:2col}.  
	
	\end{proof}
	
	\section{An open problem}
	
	A satisfying negative answer to Problem \ref{p:main} would be to establish that given a finite graph $H$, deciding the existence of an FIID homomorphism from $T_3$ to $H$ is computationally hard. In order to achieve this, one must come up with ways to ensure the existence of FIID homomorphisms and conversely, to exclude these as well. We believe that it might be possible to use entropy inequalities prove the latter part of such results. However, our knowledge about the former seems to be seriously lacking, in particular, we know only examples of graphs $H$ to which $T_3$ admits an FIID homomorphism which are ``close" to complete graphs.
	
	\begin{problem}
		Find sufficient conditions on $H$ for the existence of FIID homomorphisms from $T_d$ to $H$.  Is it true that if there is an FIID homomorphism from $T_3$ to $H$ then $H$ must contain a triangle?
	\end{problem}
	
    \bibliographystyle{abbrv}
    \bibliography{bibliography.bib}

\begin{thebibliography}{10}

\bibitem{backhausz2018entropy}
{\'A}.~Backhausz, B.~Gerencs{\'e}r, and V.~Harangi.
\newblock Entropy inequalities for factors of iid.
\newblock {\em Groups, Geometry, and Dynamics}, 13(2):389--414, 2018.

\bibitem{BGHV}
{\'A}.~Backhausz, B.~Gerencs{\'e}r, V.~Harangi, and M.~Vizer.
\newblock Correlation bounds for distant parts of factor of iid processes.
\newblock {\em Combin. Probab. Comput.}, 27(1):1--20, 2018.

\bibitem{backhausz2018large}
{\'A}.~Backhausz and B.~Szegedy.
\newblock On large-girth regular graphs and random processes on trees.
\newblock {\em Random Structures \& Algorithms}, 53(3):389--416, 2018.

\bibitem{BackVir}
{\'A}.~Backhausz and B.~Virag.
\newblock Spectral measures of factor of i.i.d. processes on vertex-transitive
  graphs.
\newblock {\em Ann. Inst. Henri Poincaré Probab. Stat.}, 53(4):2260--2278,
  2017.

\bibitem{ArankaFeriLaci}
F.~Bencs, A.~Hruskov{\'a}, and L.~M. Tóth.
\newblock Factor of iid {S}chreier decoration of transitive graphs.
\newblock {\em arXiv:2101.12577}, 2021.

\bibitem{bernshteyn2023distributed}
A.~Bernshteyn.
\newblock Distributed algorithms, the {L}ov{\'a}sz {L}ocal {L}emma, and
  descriptive combinatorics.
\newblock {\em Inventiones mathematicae}, pages 1--48, 2023.

\bibitem{bollobas1981independence}
B.~Bollob{\'a}s.
\newblock The independence ratio of regular graphs.
\newblock {\em Proceedings of the American Mathematical Society},
  83(2):433--436, 1981.

\bibitem{bowen2010measure}
L.~P. Bowen.
\newblock A measure-conjugacy invariant for free group actions.
\newblock {\em Annals of mathematics}, pages 1387--1400, 2010.

\bibitem{brandt19automatic}
S.~Brandt.
\newblock An automatic speedup theorem for distributed problems.
\newblock In {\em Proceedings of the 2019 {ACM} Symposium on Principles of
  Distributed Computing, {PODC} 2019, Toronto, ON, Canada, July 29 - August 2,
  2019}, pages 379--388, 2019.

\bibitem{brandt2021local}
S.~Brandt, Y.-J. Chang, J.~Greb{\'\i}k, C.~Grunau, V.~Rozho{\v{n}}, and
  Z.~Vidny{\'a}nszky.
\newblock Local problems on trees from the perspectives of distributed
  algorithms, finitary factors, and descriptive combinatorics.
\newblock {\em arXiv preprint arXiv:2106.02066}, 2021.

\bibitem{brandt2021homomorphism}
S.~Brandt, Y.-J. Chang, J.~Greb{\'\i}k, C.~Grunau, V.~Rozho{\v{n}}, and
  Z.~Vidny{\'a}nszky.
\newblock On homomorphism graphs.
\newblock {\em arXiv preprint arXiv:2111.03683}, 2021.

\bibitem{csoka2015invariant}
E.~Cs{\'o}ka, B.~Gerencs{\'e}r, V.~Harangi, and B.~Vir{\'a}g.
\newblock Invariant gaussian processes and independent sets on regular graphs
  of large girth.
\newblock {\em Random Structures \& Algorithms}, 47(2):284--303, 2015.

\bibitem{GamarnikSudan}
D.~Gamarnik and M.~Sudan.
\newblock Limits of local algorithms over sparse random graphs.
\newblock {\em Proceedings of the 5-th Innovations in Theoretical Computer
  Science conference, ACM Special Interest Group on Algorithms and Computation
  Theory}, 2014.

\bibitem{HarangiVirag}
V.~Harangi and B.~Vir{\'a}g.
\newblock Independence ratio and random eigenvectors in transitive graphs.
\newblock {\em Ann. Probab.}, 43(5):2810--2840, 2015.

\bibitem{HoppenWormald}
C.~Hoppen and N.~Wormald.
\newblock Local algorithms, regular graphs of large girth, and random regular
  graphs.
\newblock {\em Combinatorica}, 38(3):619--664, 2018.

\bibitem{kechris2015descriptive}
A.~S. Kechris and A.~S. Marks.
\newblock Descriptive graph combinatorics.
\newblock 2015.
\newblock \textit{Available at} \url{http://math. ucla. edu/~marks/}.

\bibitem{LyonsTrees}
R.~Lyons.
\newblock Factors of iid on trees.
\newblock {\em Combin. Probab. Comput.}, 26(2):285--300, 2017.

\bibitem{marksdet}
A.~S. Marks.
\newblock A determinacy approach to {B}orel combinatorics.
\newblock {\em J. Amer. Math. Soc.}, 29(2):579--600, 2016.

\bibitem{nesetril1999aspects}
J.~Nesetril.
\newblock Aspects of structural combinatorics (graph homomorphisms and their
  use).
\newblock {\em Taiwanese Journal of Mathematics}, 3(4):381--423, 1999.

\bibitem{RahVir}
M.~Rahman and B.~Vir{\'a}g.
\newblock Local algorithms for independent sets are half-optimal.
\newblock {\em Ann. Probab.}, 45(3):1543--1577, 2017.

\bibitem{Timar1}
{\'A}.~Tim{\'a}r.
\newblock Tree and grid factors for general point processes.
\newblock {\em Electron. Comm. Probab.}, 9(53--59).

\end{thebibliography}
\end{document}